\date{May 8, 2024}
\documentclass[12pt,reqno]{amsart}
\usepackage{latexsym,amsmath,amsfonts,amscd,amssymb}
\usepackage{graphics,color}
\makeatletter
\@namedef{subjclassname@2020}{\textup{2020} Mathematics Subject Classification}
\makeatother

\newtheorem{theorem}{Theorem}

\newtheorem{proposition}[theorem]{Proposition}
\newtheorem{corollary}[theorem]{Corollary}

\newtheorem{definition}[theorem]{Definition}
\newtheorem{conjecture}[theorem]{Conjecture}

\theoremstyle{remark}

\newcommand{\cL}{{\mathcal L}}
\newcommand{\cO}{{\mathcal O}}

\newcommand{\CC}{{\mathbb C}}

\newcommand{\KK}{{\mathbb K}}

\newcommand{\QQ}{{\mathbb Q}}
\newcommand{\RR}{{\mathbb R}}

\newcommand{\ZZ}{{\mathbb Z}}

\title{Stirling-Ramanujan constants are exponential periods}


\author[V. Mu\~{n}oz]{Vicente Mu\~{n}oz}
\address{Instituto de Matem\'atica Interdisciplinar and Departamento 
de \'Algebra, Geometr\'{\i}a y Topolog\'{\i}a, Universidad Complutense de Madrid, Spain}
\email{vicente.munoz@ucm.es}

\author[R. P\'{e}rez-Marco]{Ricardo P\'{e}rez-Marco}
\address{CNRS, IMJ-PRG, Universit\'e Paris Cit\'e, France}
\email{ricardo.perez.marco@gmail.com}

\begin{document}

\maketitle

\begin{abstract}
Ramanujan studied a general class of Stirling constants that are the resummation of some natural divergent series. 
These constants include the classical 
Euler-Mascheroni, Stirling and Glaisher-Kinkelin constants. We find natural
integral representations for all these constants that appear as exponential 
periods in the field $\QQ(t,e^{-t})$ which reveals their natural transalgebraic nature. We conjecture that all
these constants are transcendental numbers.
Euler-Mascheroni's and Stirling's integral formula are classical, 
but the integral formula for Glaisher-Kinkelin, as well as the 
integral formulas for the higher Stirling-Ramanujan constants, appear to be new. The method presented generalizes naturally to 
prove that many other constants are exponential periods over the field $\QQ(t,e^{-t})$.
\end{abstract}

% \newpage

%%%%%%%%%%%%%%%%%%%%%%%%%%%%%%%%
\section{Introduction}
%%%%%%%%%%%%%%%%%%%%%%%%%%%%%%%%

A celebrated result of De Moivre \cite{DeMoi} and Stirling \cite{Sti} is the asymptotic expansion of the 
logarithm of the factorial function, 
$$
\log s!=\sum_{k=1}^s \log k =\left (s+\frac12 \right )\log s -s +S_0 + \cO (1/s).
$$
The constant term in this asymptotic expansion is the Stirling constant 
$$
S_0=\frac{\log 2\pi}{2} =\lim_{s\to +\infty} \sum_{k=1}^s \log k -\left (\left (s+\frac12 \right )\log s -s\right ) .
$$
It is well known that the coefficients of the decaying rest $\cO (1/s)$ of the asymptotic expansion 
in the bases of the $(s^{-k})_{k\geq 1}$ 
are all rational numbers given by Bernoulli numbers. Thus, except for the Stirling constant,
all coefficients in the asymptotic expansion are rational numbers. It is not known if 
the Stirling constant is transcendental, and not even known if it is irrational\footnote[1]{The constants $e$ and $\pi$ 
are conjectured to be algebraically independent over the rationals and this easily implies that 
the Stirling constant is irrational. The algebraic independence of $e$ and $\pi$ follows from Schanuel's conjecture 
that also implies that the Stirling constant is transcendental.}.

Euler defined the Gamma function that interpolates the factorial function. The
Euler-Mascheroni constant $\gamma$ appears when looking at the logarithmic derivative of the Gamma function.
It can also be defined in a similar way as the Stirling constant by
looking at the asymptotic expansion of the divergent harmonic series
$$
\sum_{k=1}^s \frac1k = \log s +\gamma +\cO(1/s).
$$
So we have 
$$
\gamma =\lim_{s\to +\infty} \sum_{k=1}^s \frac1k  - \log s\, .
$$
Again, in this asymptotic expansion all coefficients other 
than $\gamma$ are rational numbers, the decaying part in $\cO(1/s)$ is given by Bernoulli numbers. The arithmetic 
nature of the Euler-Mascheroni constant is completely mysterious, but 
it is conjectured to be transcendental (it is not even known to be irrational). For the history of the Euler-Mascheroni constant 
we refer to the classical account by Glaisher \cite{G} and the modern one by Lagarias \cite{La}.

In general, using the Euler-McLaurin formula, Ramanujan  studied the asymptotic expansion of the 
logarithm of the $n$-factorial function (see \cite[Chapter 9, p.\ 273]{Ber})
$$
s!^{(n)}=1^{1^n}.2^{2^n}.3^{3^n}.4^{4^n} \ldots s^{s^n} \, .
$$
For example, for $n=1,2,3$ we have,
\begin{align*}
\sum_{k=1}^s k \log k &=\left ( \frac12 s^2 +\frac12 s+\frac{1}{12}\right ) \log s -\frac{1}{4}s^2 +S_1 +\cO(1/s), \\
\sum_{k=1}^s k^2 \log k &=\left ( \frac13 s^3 +\frac12 s^2+\frac{1}{6}s \right ) \log s -\frac19 s^3 +\frac{1}{12} s +S_2 + \cO(1/s), \\
\sum_{k=1}^s k^3 \log k &=\left ( \frac14 s^4 +\frac12 s^3+\frac{1}{4}s^2 -\frac{1}{120} \right ) \log s -\frac{1}{16} s^4 +\frac{1}{12} s^2 +S_3 + \cO(1/s).
\end{align*}

The Euler-McLaurin formula shows that, except for the constant coefficient $S_n$, all coefficients are 
rational numbers, in particular those of the remaining part $\cO(1/s)$ are given by Bernoulli numbers.

In general, for $n\geq 0$, we have the asymptotic expansion defining the constant $S_n$
\begin{equation*}
 \sum_{k=1}^s k^n \log k = A_n(s) \log s + B_n(s) +S_n + R_n(1/s)
\end{equation*}
where $A_n\in \QQ[s]$, $B_n\in s\QQ[s]$, and $R_n(1/s)\in  \frac1s \QQ[[\frac1s]]$.

The constants $S_n$, $n\geq 0$, are generalizations of the Stirling constant that we name, 
for historical reasons, \textit{Stirling-Ramanujan constants}.  The
Stirling constant appears as $S_0$ for $n=0$. 

For $n\leq -1$ the natural completion of the family of constants is not what we may expect. 
For $n\leq -1$ we define $S_n$ as the 
constant in the asymptotic expansion of the sum $\displaystyle  \sum_{k=1}^s k^n$. The reason for considering for $n\leq -1$ this series instead of the 
series $\displaystyle  \sum_{k=1}^s k^n \log k$ is that we consider sums of the form $\displaystyle  \sum_{k=1}^s f_n(k)$ and, for 
$n\leq 0$, we will have $f_n' =n f_{n-1}$, with $f_n(k)=k^n$ for $n\leq -1$, and starting with  $f_0(n)=1+\log n$. 
On the other hand, for $n\geq 1$, we have $f_n' =n f_{n-1}+k^{n-1}$. Note that the additive 
monomial term only introduces a simple polynomial additive term (without constant term) 
in the asymptotics of the sums that we study. So the constant in the asymptotic expansion 
stays the same. 
For all $n\in \ZZ$ we have $f_n' =n f_{n-1}$ modulo a polynomial, i.e.
we generate the sequence of functions $(f_n)$ by iterating primitives. Why this is natural becomes clear in Section 2 where we generate 
the higher Frullani integrals
as iterated integrals\footnote[2]{Ramanujan also studied the other closely related constants obtained by resummation of series 
$\displaystyle  \sum_{k=1}^s k^n \log k$ for $n<0$ (see \cite{Ber} Section 8) that are converging for $n\leq -2$ and equal to $-\zeta(-n)$. For the 
relation with Ramanujan resummation in this converging case see section V.5.7 p.168 in \cite{Ca2} (we thank the referee for this reference).}.
For $n=-1$ we obtain the Euler-Mascheroni 
constant $\gamma=S_{-1}$ which corresponds to the divergent harmonic series. For $n\leq -2$ the series is converging to $S_n=\zeta(-n)$.

The constant $S_1$, more precisely $A=e^{S_1}$, appears in the works of 
Glaisher \cite{G} and Kinkelin \cite{Ki}. It is known as the 
Glaisher-Kinkelin constant and arises in the theory of the Barnes Gamma function \cite{Ba}.

Again, all the coefficients in the above asymptotic expansions other than the constant 
one are rational numbers. The arithmetic nature of the
Stirling-Ramanujan constants is unknown. It is natural to conjecture: 

\begin{conjecture}
The Stirling-Ramanujan constants are transcendental. 
\end{conjecture}

The conjecture is known for $n$ negative even integers, but otherwise open. The conjecture for $n=-1$ is the 
conjectured transcendence of Euler-Mascheroni constant.

These generalized Stirling constants had a deep meaning for Ramanujan transalgebraic view. 
For him, they are the ``sum'' of the divergent infinite series by the resummation method he 
discovered based on the Euler-McLaurin formula. 
As he puts it, these are ``the barycenter of the divergent sum'' that balances between the divergent series 
and the divergent integral in the Euler-McLaurin summation formula \cite{Ha1}. 
Ramanujan defines and computes these constants using the derivative 
of the Riemann zeta function at positive integer values and the values of the zeta function at positive odd values 
(see Chapter 9 in \cite{Ber}, p. 273-276, in particular formula (27.7) and the Corollary in p.276).
Ramanujan resummation procedure is described (in an imperfect form) in the classical book of Hardy \cite{Ha2}.
The reader can directly consult the original source for Ramanujan's resummation in Berndt volumes (in particular, chapter 6 of \cite{Ber}). 
More recently, a rigorous approach is proposed in the
monographs by Candelpergher \cite{Ca1}, \cite{Ca2}. Another rigorous approach is based on the Poisson-Newton formula   \cite{MPM2} 
was found, but not yet fully developed, by the authors (see Section 5.5, p.33 in \cite{MPM1}).

The Stirling-Ramanujan constants reappeared after Ramanujan in 1933 when Bendersky \cite{Be} defined 
a new hierarchy of higher Gamma functions 
different from that of Barnes \cite{Ba}. The Bendersky Gamma function interpolates the $n$-factorial function.
Contrary to the Barnes Gamma functions, higher Bendersky Gamma functions are not meromorphic since they have 
ramified singularities at negative integers. Indeed, historically, works of Kinkelin \cite{Ki} and Alekseevsky \cite{Ale1} \cite{Ale2} 
preceded by many decades those of Barnes and Bendersky. The reader can consult the excellent recent historical account by Neretin \cite{Ne} 
on the precursor work of Alekseevsky.

The Stirling-Ramanujan constants have been rediscovered multiple times by several authors, unaware that they 
appear already in Ramanujan notebooks (for example, see constants 
$B$ and $C$ in \cite[p.\ 56]{SC} and the comments therein). In some recent literature these constants are named \textit{Bendersky-Adamchik constants}
(in \cite{Ad} Adamchik rediscovers Ramanujan's computation of these constants in terms of values of the derivative of the zeta function 
at integer values). 

In a recent article \cite{Co}, Coppo gives remarkable converging series for the constants. Another
noteworthy series representation of the Glaisher-Kinkelin constant is given by Pain in \cite{Pa}.

All these higher Gamma functions and the associated Stirling-Ramanujan constants arise in multiple mathematical contexts, as for 
example in Shintani's work in Algebraic Number theory, in the computation of determinants of Laplacians in Riemannian geometry, 
and in the zeta-regularization theory in Physics. As a curiosity, in the article \cite{KB} in the Physics literature we 
find the  computation of the Stirling-Ramanujan constants up to $S_3$ (probably the record in the literature).

% \newpage

The main purpose of this article is to give a general whole integral formulas for 
all the Stirling-Ramanujan constants $S_n$.

\begin{theorem}[General Exponential Period Formula] \label{thm:main}
% \break

For $n\geq 0$ we have
$$
S_n =  (-1)^{n+1} n! \int_0^{+\infty}  
 \frac{1}{t^n} \left (\frac{1}{1-e^{-t}}- \sum_{k=-1}^{n} b_k t^{k} -r_n t^{n+1}\right ) \frac{e^{-t} dt}{t} \, ,
$$
where $r_n$ is the rational number
$$
r_n=\sum_{k=1}^{n+1} \frac{(-1)^{k+1}}{k!} \, b_{n-k}  H_{k}\, ,
$$
where the $(H_k)$ are the harmonic numbers, $H_k=1+\frac12+\ldots +\frac1k$, and 
the coefficients $(b_k)$ have the generating function
$$
 \frac{1}{1-e^{-t}}= \sum_{k=-1}^{+\infty} b_k t^k\, ,
$$
and they are given by $b_k=(-1)^{k+1}\frac{B_{k+1}}{(k+1)!}$, where $B_k$ is the $k$-th Bernoulli number.
\end{theorem}

There are multiple different expressions for the Euler-Mascheroni, Stirling and Glasher-Kinkelin constants
in the literature (see \cite{SC} for example). But starting from $S_1$ these integral expressions seem to be new.

For $n=-1$ the formula makes sense and we recover the classical integral formula for $\gamma=S_{-1}$  
\begin{equation*}
\gamma = \int_0^{+\infty} t \, \left(\frac{1}{1-e^{-t}}-\frac1t \right) \, \frac{e^{-t}}{t} \, dt \, .
\end{equation*}
This integral formula is due to Euler back in 1770 (\cite{[E393]} section 25), but Whittaker and Watson \cite[Example 2, p.\ 248]{WW} 
attributes this formula to Dirichlet (presumably to \cite{Di}, 1836). Moreover, Euler also gives 
this formula in \cite{[E583]} (1785) and devotes a full article to it (\cite{[E629]}, 1789).

For $n\leq -2$ the formula also makes sense. In that case the series 
$$
\zeta (-n)=\sum_{k=0}^{+\infty} (k+1)^n
$$
is converging, the asymptotic expansion of the partial sum reduces to a constant (its sum),  
and the integral formula is directly related to the classical formula for the Riemann zeta function used 
by Riemann in his investigations of the complex extension of $\zeta$
$$
\zeta(s) =\frac{1}{\Gamma(s)} \, \int_0^{+\infty} \frac{t^{s}}{1-e^{-t}} \, \frac{e^{-t} dt}{t} \, .
$$

For $n=0$, for the Stirling constant $S_0$ we recover the classical integral formula due to Pringsheim \cite{Pr} (see also \cite{WW}, \cite[p.\ 288]{Pa})
$$
S_0=\frac{\log (2 \pi )}{2} = - \int_0^{+\infty} \left(\frac{1}{1-e^{-t}}-\frac1t -\frac12 -t \right) \, \frac{e^{-t}}{t} \, dt \, .
$$

For $n=1$, for the Glaisher-Kinkelin constant $\log A$ we get the following integral formula:
\begin{corollary}[Integral formula for Glaisher-Kinkelin constant]
\begin{equation*}
  S_1 = \log A=\int_0^{+\infty} \frac1t \, \left(\frac{1}{1-e^{-t}} - \frac1t
   -\frac{1}{2}-\frac{t}{12}+ \frac{t^2}{4}\right) \frac{e^{-t}}{t} \,  dt \, .
   \end{equation*}
 \end{corollary}

We also get the following new integral formulas for the constants $S_2$ and $S_3$.

\begin{corollary} We have
\begin{align*}
S_2 &=-2\int_0^{+\infty} \frac{1}{t^2} \, \left(\frac{1}{1-e^{-t}} - \frac1t
   -\frac{1}{2}-\frac{t}{12}   -\frac1{72} t^3 \right) \frac{e^{-t}}{t} \,  dt  \\
S_3 &= 6 \int_0^{+\infty} \frac{1}{t^3} \, \left(\frac{1}{1-e^{-t}} - \frac1t
   -\frac{1}{2}-\frac{t}{12}  +\frac{t^3}{720} -\frac{1}{288} t^4\right) \frac{e^{-t}}{t} \,  dt 
\end{align*}
\end{corollary}

An interesting fact about all these new formulas is that all the Stirling-Ramanujan constants appear 
as Exponential Periods with base field of functions $\QQ(t,e^{-t})$. This reveals their  
transalgebraic nature. In general, we can consider a class $\KK$ of transalgebraic functions, as those appearing in 
Liouville towers of functions (see for example \cite{Kho}, \cite{Ritt}), obtained by algebraic, logarithmic, integral 
and exponential extensions.
One defines a \textit{Transalgebraic period} over $\KK$ is a number of the form
$$
\omega = \int_\eta f(s) \, ds 
$$
where the integral is converging and taken over an open path $\eta$ where $f\in \KK$ is holomorphic and 
$\eta$ has end-points that are singularities of $f$.
This notion is interesting particularly when the class of functions $\KK$ is a countable class, then 
most of complex numbers are not periods over $\KK$. 
For example this occurs for a Liouville type of extension of the differential field 
of rational functions $\QQ(s)$. When we adjoint a finite number of exponential function $\KK=\QQ(t,e^{\omega_1 t}, \ldots, e^{\omega_nt})$, 
we obtain \textit{Exponential Periods}.

A different notion of exponential periods is considered by Konstevich and Zagier \cite{KZ} by analogy
with the more classical algebraic periods, where an algebraic differential form 
is integrated over a cycle of an algebraic variety. It is easy to see that this definition is 
more restrictive than Transalgebraic Periods as defined above.
Our point of view follows the classical origin.
Exponential periods were studied by Liouville, Picard, Fuchs,... already in the 19th century. 
They appear in the natural context of log-Riemann surface theory as the location of the ramifications, and 
as asymptotic values of transcendental functions  
associated to log-Riemann surfaces. These  log-Riemann surfaces are transalgebraic curves that generalize classical algebraic 
curves by allowing a finite number of infinite ramification points. The determinant of the matrix of periods which comes from 
a bases of a vector space of transcendental functions associated to the transalgebraic curve is the Ramificant Determinant.
This determinant is a fundamental object of the transalgebraic algebra of transalgebraic curves 
(see the original manuscript \cite{BPM1}, and subsequent articles \cite{BPM2}, \cite{BPM3}, and in particular \cite{BPM4} where some 
exponential periods are the main object of study, and the more recent one \cite{PM2}).

\bigskip

If we remove the combinatorial complexity in the definition and the constants, up to a rational affine combination, the family of 
Stirling-Ramanujan constants is rationally equivalent to the family of Upsilon constants considered by the authors in 
a related context (unplublished manuscript):

\begin{definition}[Upsilon constants]
For $n\geq -1$ we define
$$
\Upsilon_n = \int_0^{+\infty} \frac{1}{t^n} \left (\frac{1}{1-e^{-t}}- \sum_{k=-1}^{n} b_k t^{k} \right ) \frac{e^{-t} dt}{t}\, .
$$
\end{definition}

Obviously we have 
$$
S_n= (-1)^{n+1} n! (\Upsilon_n -r_n),
$$
and the transcendence conjecture for $n\geq -1$ is equivalent to the transcendence conjecture for the Upsilon constants. 
\begin{conjecture}
Upsilon constants $\Upsilon_n$ are transcendental. 
\end{conjecture}
It is natural to raise the question of their algebraic independence for $n\geq -1$.

% \bigskip

\newpage
% 

%%%%%%%%%%%%%%%%%%%%%%%%%%%%%%%%
\section{Higher Frullani integrals}
%%%%%%%%%%%%%%%%%%%%%%%%%%%%%%%%

The classical Frullani integral expresses $\log s$ as an exponential period:
$$
\log s = \int_0^{+\infty} \frac{1}{t}(e^{-t}-e^{-st}) dt .
$$ 
This formula takes a nicer form when we single out the differential element $\frac{e^{-t} dt}{t}$, that is the
natural differential\footnote[3]{For the reader knowledgeable of the theory of log-Riemann surfaces, 
the geometry associated to this basic differential $\omega = \frac{e^{-t} dt}{t}$ 
is the tube-log-Riemann surface that derives from 
the log-Riemann surface of the logarithm replacing a planar sheet by a tube $\CC/\ZZ$ 
(see \cite{BPM1}).} in $\CC^*$  for periods with an exponential singularity at $\infty$ and a polar 
singularity at $0$.
The expressions are also nicer when we work with functions of the variable $s+1$ instead of $s$. 
This is because $\log (s+1)$, 
contrary to $\log (s)$, is an LLD function (a Left Located Divisor function has all of its zeros and poles, or in 
general singularities, in the left half plane $\{ \Re s<0\}$). This is a relevant notion in the theory 
of Poisson-Newton formula (\cite{MPM1}, \cite{MPM2}, \cite{MPM3}).

Because of these reasons, the natural way to write down the Frullani integral is
$$
\log (s+1) = \int_0^{+\infty} (1-e^{-st}) \frac{e^{-t} dt}{t} \ .
$$
Observe that the Frullani integral is obtained by integration on the variable $s$ over the 
interval $[0,s]$ of the elementary 
integral
$$
\frac{1}{s+1}=\int_0^{+\infty} t \, e^{-st} \, \frac{e^{-t} dt}{t} \ ,
$$
that can be thought of the primitive generator integral.
Starting with the Frullani integral and iterating integrations on the variable $s$ over 
the interval $[0,s]$, we obtain higher Frullani integrals that express $s^k\log s$ as an 
exponential integral.

\begin{theorem} \label{thm:Frullani}
We have 
\begin{align*}
(s+1)^n  \log (s+1) 
&= \sum_{k=1}^n \binom{n}{k} (H_n-H_{n-k})s^k + \\
&\ \ \ +(-1)^{n+1} n! \int_0^{+\infty} \frac{1}{t^n}\left(e^{-st}- \sum_{k=0}^{n} \frac{(-s)^k t^k}{k!} \right) 
\frac{e^{-t} dt}{t}\, ,
\end{align*}
where  $H_n=1+\frac12+\cdots +\frac1n$ are the harmonic numbers for $n\geq 1$, and $H_0=0$.
\end{theorem}

This explicit formula can be found in the references \cite{M} and \cite{MMR}. These higher Frullani 
integrals are relevant to explain the most popular BBP formulas as was unveiled in the recent article by D.~Barsky and the authors 
\cite{BMPM}. A complete proof of Theorem \ref{thm:Frullani} can be found in this reference (Prop. 2.1 and Prop. 2.2 of \cite{BMPM}).

We have a polynomial part, without constant term, and an integral part which 
is the Hadamard regularization of the Laplace transform of $\frac{1}{t^{n+1}} e^{-t}$
as considered in \cite{MPM2}.

%%%%%%%%%%%%%%%%%%%%%%%%%%%%%%%%%%%
\section{Technical preliminaries}
%%%%%%%%%%%%%%%%%%%%%%%%%%%%%%%%%%%

\subsection{Technical normalizations} 
It appears combinatorially simpler to find the integral formulas for the constants appearing in a slight
different expansion base, which is closer to the Upsilon constants from the introduction.

In the proof it is natural to compute the expansion on the shifted variable $s+1$,
and this comes down to 
$$
\sum_{k=1}^{s} k^n \log k =\tilde A_n(s) \log (s+1) + \tilde B_n(s) +\tilde S_n + \tilde R_n(1/s),
$$
where $\tilde A_n \in \QQ[s]$,  $\tilde B_n\in s\QQ[s]$, $\deg \tilde A_n =\deg \tilde B_n=n+1$, 
and $\tilde R_n(1/s)\in \frac1s\QQ[[\frac1s]] $. 

From this it follows the expansion in the form used to define the Stirling-Ramanujan constants
$$
\sum_{k=1}^s k^n \log k =A_n(s) \log s + B_n(s) +S_n + R_n(1/s),
$$
with $A_n\in \QQ[s]$, $B_n\in s\QQ[s]$, and $R_n(1/s)\in  \frac1s\QQ[[\frac1s]] $.

Indeed, we have
\begin{align*}
\sum_{k=1}^s k^n \log k &= \tilde A_n(s) \log (s+1) + \tilde B_n(s) +\tilde S_n + \tilde R_n(1/s) \\
&= \tilde A_n(s) \log (1+1/s) + \tilde A_n(s) \log s  + \tilde B_n(s) +\tilde S_n + \tilde R_n(1/s) \\
 &=\tilde A_n(s) \sum_{k=1}^{+\infty} \frac{(-1)^{k+1}}{ks^k}   + \tilde A_n(s) \log s  + \tilde B_n(s) +\tilde S_n + \tilde R_n(1/s) .
\end{align*}
Therefore, if we write down the polynomial as $\tilde A_n(s) = \sum\limits_{j=0}^{n+1} \tilde a_j s^j$,  then we have
\begin{equation} \label{eqn:move}
S_n =\tilde S_n + \sum_{k=1}^{n+1} \frac{ (-1)^{k+1}}{k} \, \tilde a_k  \, .
\end{equation}

\subsection{Asymptotic decay}

We use the following classical and well known fact from Laplace transform theory (see \cite{DeBru} for example).

\begin{proposition}\label{prop:Laplace_decay}
Let $f: \RR_+\to \RR$ differentiable, in particular at the right at $t=0$. We assume $f'$ is bounded. Then the 
Laplace transform
$$
{\cL} f(s)=\int_0^{+\infty} f(t)  \, e^{-st}\, dt
$$
decays as $\cO(1/s)$ when $s\to +\infty$. 
Moreover, if we assume that $f$ is infinitely differentiable and the derivatives are bounded, then 
for every $n\geq 1$, we have the asymptotic expansion
$$
{\cL} f(s)= \frac{f(0)}{s} +\frac{f'(0)}{s^2}+\ldots +\frac{f^{(n-1)}(0)}{s^n} +\cO(1/s^{n+1}).
$$
\end{proposition}

\begin{proof}
The proof is immediate by integration by parts:
\begin{align*}
{\cL} f(s)&= \left [-f(t) \frac{e^{-st}}{s}\right ]_0^{+\infty} +\frac{1}{s} \int_0^{+\infty} f'(t)  \, e^{-st}\, dt \\
&=\frac{f(0)}{s}+\frac{1}{s} \int_0^{+\infty} f'(t)  \, e^{-st}\, dt \\
&=\frac{f(0)}{s}+\frac{1}{s} \cL f'(s).
\end{align*}
As $f'(s)$ is bounded, we have that $\cL f'(s)$ is bounded for $s\geq 1$. This implies that $\cL f(s)=\cO(1/s)$,
and the same will be true for $f'$, that is, $\cL f'(s)=\cO(1/s)$, hence $\cL f(s)=\frac{f(0)}{s} + \cO(1/s^2)$.
Working inductively, we get the result for any $n\geq 1$.
\end{proof}

%%%%%%%%%%%%%%%%%%%%%%%%%%%%%%%%%%%
\section{Derivation of the integral formulas}
%%%%%%%%%%%%%%%%%%%%%%%%%%%%%%%%%%%

Consider the $n$-th higher Frullani integrals given in Theorem \ref{thm:Frullani}
and we add from $s=0$ to $s-1$ to get (we use the convention $0^0=1$)
\begin{align*}
\sum_{k=1}^{s} k^n \log k = &\sum_{k=1}^n \binom{n}{k} (H_n-H_{n-k}) \left ( \sum_{j=0}^{s-1} j^k \right ) + \\
& + (-1)^{n+1} n! \int_0^{+\infty} \frac{1}{t^n}\left(\frac{1-e^{-st}}{1-e^{-t}} - 
\sum_{k=0}^{n} \frac{ (-t)^k}{k!} \left ( \sum_{j=0}^{s-1} j^k  \right ) \right) 
\frac{e^{-t} dt}{t}\, .
\end{align*}

By Faulhaber's formula, for $k\geq 1$,
$$
\sum_{j=0}^{s-1} j^k 
$$
is a polynomial in the variable $s$ without constant term, whose coefficients 
are given by Bernoulli numbers (we remind this fact in the Appendix).

Therefore, we have that the sum
$$
\sum_{k=1}^n \binom{n}{k} (H_n-H_{n-k}) \left (\sum_{j=0}^{s-1} j^k \right )
$$
is also a polynomial in the variable $s$ without constant term and we can disregard it when looking for  
the constant coefficient of the asymptotic expansion for $s\to +\infty$.

Thus, we reduce the problem at looking for the constant term of the asymptotic expansion of the integral
\begin{align*}
I_n &=(-1)^{n} n!  \int_0^{+\infty} \frac{1}{t^n}\left(\frac{e^{-st}-1}{1-e^{-t}} + 
\sum_{k=0}^{n} \frac{ (-t)^k}{k!} \left ( \sum_{j=0}^{s-1} j^k  \right ) \right) 
\frac{e^{-t} dt}{t} \\
 &= (-1)^{n} n!   \int_0^{+\infty} \left( 
e^{-st}\frac{1}{t^n}
\frac{1}{1-e^{-t}} -
\frac{1}{t^n}\frac{1}{1-e^{-t}}+ \frac{1}{t^n}\sum_{k=0}^{n} \frac{ (-t)^k}{k!} \left (\sum_{j=0}^{s-1} j^k  \right ) 
\right) 
 \frac{e^{-t} dt}{t} \, .
\end{align*}

We use the expansion 
$$
\frac{1}{1-e^{-t}}=\sum_{k=-1}^{+\infty} b_k t^k =\frac1t +\frac12+\frac{1}{12}t -\frac{1}{720} t^3+\ldots
$$
Then
\begin{align}
&\frac{(-1)^{n}}{ n!}  I_n = 
\int_0^{+\infty} \Bigg ( e^{-st} \sum_{k=-1}^{n} b_k t^{k-n}+
e^{-st} \frac{1}{t^n} \left (\frac{1}{1-e^{-t}}- \sum_{k=-1}^{n} b_k t^{k} \right ) \nonumber \\
& \qquad \qquad \qquad\qquad \qquad \qquad -\frac{1}{t^n}\frac{1}{1-e^{-t}}+ \frac{1}{t^n} \sum_{k=0}^{n} \frac{(-t)^k}{k!} 
\left (\sum_{j=0}^{s-1} j^k  \right ) \Bigg) \frac{e^{-t} dt}{t} \nonumber \\
&= \sum_{k=-1}^{n} b_k  \int_0^{+\infty} \frac{1}{t^{n-k}}\left (e^{-st} -
\sum_{l=0}^{n-k} \frac{(-st)^l}{l!} \right ) \frac{e^{-t} dt}{t} + \label{eqn:1}\\
&\ \  + \int_0^{+\infty}  
e^{-st} \frac{1}{t^n} \left (\frac{1}{1-e^{-t}}- \sum_{k=-1}^{n} b_k t^{k} \right ) \frac{e^{-t} dt}{t} + \label{eqn:2}\\
& \ \ + \int_0^{+\infty} \left [  \sum_{k=-1}^{n} \frac{b_k}{t^{n-k}}   \sum_{l=0}^{n-k} \frac{(-st)^l}{l!}  
-\frac{1}{t^n} \frac{1}{1-e^{-t}}+ \frac{1}{t^n} \sum_{k=0}^{n} \frac{ (-t)^k}{k!} \left ( \sum_{j=0}^{s-1} j^k  \right ) \right ]
\frac{e^{-t} dt}{t}\, . \label{eqn:3}
\end{align}

In line (\ref{eqn:1}) we recognize the first $n+1$ higher Frullani integrals:
\begin{align*}
\sum_{k=-1}^{n} b_k  \int_0^{+\infty} & \frac{1}{t^{n-k}}\left (e^{-st} -
\sum_{l=0}^{n-k} \frac{(-st)^l}{l!} \right ) \frac{e^{-t} dt}{t}  \\
 &= \left (\sum_{k=-1}^{n} b_k  \, \frac{(-1)^{n-k+1}}{(n-k)!} (s+1)^{n-k}\right ) \, \log (s+1) \quad   (\text{mod} \, s\QQ[s] )  \\
&=  \frac{(-1)^n}{ n!} \, \tilde A_n(s) \log (s+1) \quad  (\text{mod} \, s\QQ[s] ),
\end{align*}
for some polynomial $\tilde A_n(s)\in \QQ[s]$. 
Note that the coefficients $\tilde a_k$ of 
the polynomials $\tilde A_n$ can be computed explicitly from the coefficients $(b_k)_{k\geq -1}$. Indeed
for $0\leq j\leq n+1$, we have
\begin{equation}\label{eqn:aj}
 \tilde{a}_j= \sum_{k=-1}^{n-j}  b_k (-1)^{k+1} \frac{n!}{(n-k-j)!j!}\, .
 \end{equation}

The second integral (\ref{eqn:2}) is $\cO(1/s)$ when $s\to +\infty$ using Proposition \ref{prop:Laplace_decay}.

For the third integral (\ref{eqn:3}), we replace the sum of $k$-powers of the first $s-1$ integers by the Faulhaber polynomial in $s$ that is given 
by the Bernoulli type coefficients $(b_k)$ (use Proposition \ref{prop:Faulhaber} from the Appendix),
\begin{align*}
&\int_0^{+\infty} \left [  \sum_{k=-1}^{n} \frac{b_k}{t^{n-k}}   \sum_{l=0}^{n-k} \frac{(-st)^l}{l!}  
-\frac{1}{t^n} \frac{1}{1-e^{-t}}+ \frac{1}{t^n} \sum_{k=0}^{n} \frac{ (-t)^k}{k!} \left ( \sum_{j=0}^{s-1} j^k  \right ) \right ]
\frac{e^{-t} dt}{t}\\
&= \int_0^{+\infty} \left [  \sum_{h=-1}^{n} \frac{b_h}{t^{n-h}}   \sum_{l=0}^{n-h} \frac{(-st)^l}{l!}  
-\frac{1}{t^n} \frac{1}{1-e^{-t}}+ \frac{1}{t^n} \sum_{k=0}^{n} \frac{ (-t)^k}{k!} 
\left ((-1)^{k+1} k! \sum_{l=1}^{k+1}  (-1)^l b_{k-l} \, \frac{s^l}{l!}  \right ) \right ]
\frac{e^{-t} dt}{t} \\
& = \int_0^{+\infty}  
 \frac{1}{t^n} \left (\sum_{k=-1}^{n} b_k t^{k} -\frac{1}{1-e^{-t}} \right ) \frac{e^{-t} dt}{t} \, .
\end{align*}
The simplification to the last result comes from the observation that there is a cancellation of the terms 
in the second line for the first summation when $h$ equals to $k-l$. 
In the first sum, the condition is $0\leq l\leq n-h \neq n+1$. In the second sum
the condition is $1\leq l \leq k+1\leq n+1$. This means that only the terms for $l=0$ in the first sum remain. 
This massive cancellation is to be expected since the coefficients of the monomials on the variable $s$ would give divergent 
integrals if they wouldn't vanish.

Therefore, this gives the modified Stirling-Ramanujan constant 
$$
\tilde S_n =  (-1)^{n+1} n! \int_0^{+\infty}  
 \frac{1}{t^n} \left (\frac{1}{1-e^{-t}}- \sum_{k=-1}^{n} b_k t^{k} \right ) \frac{e^{-t} dt}{t} \, .
$$
Hence using (\ref{eqn:move}) and (\ref{eqn:aj}), we get
$$
S_n= \tilde S_n  + \sum_{j=1}^{n+1} \frac{ (-1)^{j+1}}{j} \, \tilde a_j  =  \tilde S_n  + 
\sum_{j=1}^{n+1} \frac{ (-1)^{j+1}}{j} \sum_{h=-1}^{n-j}  b_h  (-1)^{h+1}\frac{n! }{(n-h-j)!j!}\, ,
$$
and $S_n= \tilde S_n  + \hat r_n$ where $\hat r_n$ is the rational number
$$
\hat r_n =\sum_{j=1}^{n+1} \sum_{h=-1}^{n-j} \frac{ (-1)^{j+h}}{j}  b_h \frac{n! }{(n-h-j)!j!} \, .
$$
We can insert this rational number inside the integral and we finally get
$$
S_n =  (-1)^{n+1} n! \int_0^{+\infty}  
 \frac{1}{t^n} \left (\frac{1}{1-e^{-t}}- \sum_{k=-1}^{n} b_k t^{k} -r_n t^{n+1}\right ) \frac{e^{-t} dt}{t} \, ,
$$
where the rational number $r_n$ is given by
$$
r_n =\frac{ (-1)^{n}}{n!} \hat r_n = \sum_{j=1}^{n+1} \sum_{h=-1}^{n-j} \frac{ (-1)^{n+j+h}}{j}  b_h \frac{1}{(n-h-j)!j!}\, .
$$
We can simplify this expression:

\begin{proposition}\label{eqn:rn}
We have
$$
r_n=\sum_{h=-1}^{n-1} b_h  \frac{(-1)^{n-h+1}}{(n-h)!} H_{n-h}  \, .
$$
\end{proposition}

\begin{proof}
We compute 
\begin{align*} 
r_n &= \sum_{j=1}^{n+1} \sum_{h=-1}^{n-j} \frac{ (-1)^{n+j+h}}{j}  b_h \frac{1}{(n-h-j)!j!}  \nonumber\\
&= \sum_{h=-1}^{n-1} b_h  \frac{(-1)^{n-h}}{(n-h)!} \sum_{j=1}^{n-h} \frac{(-1)^{j}}{j} \binom{n-h}{j} \nonumber\\
&= \sum_{h=-1}^{n-1} b_h  \frac{(-1)^{n-h+1}}{(n-h)!} H_{n-h}\, , 
\end{align*}
where we use the identity $\sum\limits_{k=1}^n \frac{(-1)^k}{k} \binom{n}{k} =-H_n$. These type of identities of harmonic 
numbers are immediate when we write them down as exponential periods:
$$
H_n= 1+\frac12 +\ldots +\frac1n = \sum_{k=1}^n \int_0^{+\infty} e^{-(k-1)t} \, e^{-t} \, dt = \int_0^{+\infty} \frac{1-e^{-nt}}{1-e^{-t}} \, e^{-t} \, dt \, ,
$$
and then
\begin{align*}
H_n &=  \int_0^{+\infty} \frac{1-(1- (1-e^{-t}))^n}{1-e^{-t}} \, e^{-t} \, dt \\
&=- \int_0^{+\infty} \sum_{k=1}^n \binom{n}{k} (-1)^k (1-e^{-t})^{k-1} \, e^{-t} \, dt \\
&=- \sum_{k=1}^n  \binom{n}{k} (-1)^k \int_0^{+\infty}  (1-e^{-t})^{k-1} \, e^{-t} \, dt \\
&=- \sum_{k=1}^n  \binom{n}{k} \frac{(-1)^k}{k+1}  \, .
\end{align*}
\end{proof}

This completes the proof of Theorem \ref{thm:main}.

\section{Particular cases}
Using Proposition \ref{eqn:rn}, we obtain
 \begin{align*}
  r_0 &=b_{-1}=1, \\
 r_1 &= - \frac{H_2 b_{-1}}{2!} +\frac{H_1 b_0}{1!} = -\frac14, \\[2pt]
 r_2 &= \frac{H_3 b_{-1}}{3!} -\frac{H_2 b_0}{2!} +\frac{H_1 b_1}{1!} = \frac1{72}, \\[2pt]
r_3 &= -\frac{H_4 b_{-1}}{4!} +\frac{H_3 b_0}{3!} -\frac{H_2 b_1}{2!} = \frac{1}{288}\, .
\end{align*}

Therefore Theorem \ref{thm:main} gives
\begin{align*}
S_0 &=\frac{\log (2\pi)}{2}= \int_0^{+\infty} \left (\frac1t +\frac12 +t -\frac{1}{1-e^{-t}} \right ) \frac{e^{-t}}{t}  dt\, , \\
S_1 &= \log A=\int_0^{+\infty} \frac1t \, \left(\frac{1}{1-e^{-t}} - \frac1t
   -\frac{1}{2}-\frac{t}{12}+ \frac{t^2}{4}\right) \frac{e^{-t}}{t} \,  dt \, , \\
S_2 &=-2\int_0^{+\infty} \frac{1}{t^2} \, \left(\frac{1}{1-e^{-t}} - \frac1t
   -\frac{1}{2}-\frac{t}{12} -\frac1{72} t^3 \right) \frac{e^{-t}}{t} \,  dt \, , \\
S_3 &=6 \int_0^{+\infty} \frac{1}{t^3} \, \left(\frac{1}{1-e^{-t}} - \frac1t
   -\frac{1}{2}-\frac{t}{12} 
   +\frac{t^3}{720} - \frac{1}{288} t^4\right) \frac{e^{-t}}{t} \,  dt \, .
\end{align*}

The numbers $S_n$ up to $n=4$ are computed by Bendersky \cite[p.\ 265]{Be}. Translating into natural logarithm (Bendersky uses
decimal logarithm), we have
\begin{align*}
S_0 &= \log(10) \cdot 0,399 089 \ldots = 0.918938\ldots \\
S_1 &= \log(10) \cdot0.108 032 \ldots = 0.248754\ldots \\
S_2 &= \log(10) \cdot0.013 223 \ldots = 0.030448\ldots\\
S_3 &= \log(10) \cdot(-1+0.991 029 \ldots)= -0.020656\ldots
\end{align*}
This agrees with the numerical values of the  integrals in Theorem \ref{thm:main} given above.

For $n=-1$, we also have that the above formulas give $r_{-1}= 0$ and 
$$
S_{-1}= \gamma=\int_0^{+\infty} t \left (\frac{1}{1-e^{-t}} -\frac1t \right ) \frac{e^{-t}}{t} dt \, .
$$

% \newpage

\section{General scope of the method}

The precedent procedure is very general and allows to find a multitude of integral formulas for many constants defined 
as the constant term in asymptotic expansions of arithmetic nature. Consider sums (typical divergent) of the type
$$
\sum_{k=0}^{s} f(k+1),
$$
where $f(s+1)$ is an LLD function (LLD means Left Located Divisor) so that the general divisor of $f$ are in the left half plane.
When $f$ is meromorphic (for example as the Euler Gamma function $\Gamma$), 
then the general divisor is only composed by zeros and poles. In general,
we can have ramified singularities (as for the logarithm function $\log s$).

We consider the asymptotic when $s\to +\infty$ of the above sum. 
For natural arithmetic functions, these sums are divergent and have a natural asymptotic expansion on a natural bases of functions
and a decaying part in $\cO(1/s)$. The constant term in the expansion define generalized Stirling-Ramanujan constants.

A particularly interesting case is when $f$ admits a Malmst\'en type formula. We recall that Malmst\'en formula for the Euler 
Gamma function (\cite{Mal}), for $\Re s >0$,
\begin{equation*}
   \log \Gamma (s+1) = \int_0^{+\infty} \left ( s+ \frac{e^{-st}-1}{1-e^{-t}}\right ) \frac{e^{-t} dt}{t} \, .
\end{equation*}
A general Malmst\'en formula for $g=e^f$ over the ring $\QQ(t,e^{-t})[e^{-st}]$ is, for $\Re s >0$,
$$
\log g(s+1)= f(s+1)=\int_0^{\infty} G(t, e^{-t}, e^{-st}) \frac{e^{-t} dt}{t}\, .
$$
Note that this means that $f(s)$ behaves as an exponential period (with parameter $s$). 
The higher Frullani integrals are just Malmst\'en formulas 
for the $f(s)=s^n\log(s)$, or $g(s) =s^{s^n}$.

In that case the general procedure presented will also work. The key property 
is that the sum for which we study the asymptotic is also 
an exponential period in the same ring $\QQ(t,e^{-t})[e^{-st}]$.
   
% \newpage

\section{Appendix: Bernoulli computations.}

We review in this Appendix the computations with Bernoulli numbers that we use.

The classical Bernoulli polynomials $(B_k(s))_{k\geq 0}$ can be defined via their generating function
$$
\frac{te^{st}}{e^t-1} =\sum_{k=0}^{+\infty} B_k(s) \, \frac{t^k}{k!}\, .
$$
Making the change of variable $t\to -t$ this can also be written (in a form that is more appealing to us) as
$$
\frac{te^{-st}}{1-e^{-t}} =\sum_{k=0}^{+\infty} (-1)^k B_k(s) \, \frac{t^k}{k!}\, .
$$
The Bernoulli numbers are defined as $B_k=B_k(0)$.

\medskip

We prefer to work with a related sequence of polynomials $(b_k(s))_{k\geq 0}$ defined via their generating function
$$
\frac{e^{-st}}{1-e^{-t}} =\frac1t +\sum_{k=0}^{+\infty} b_k(s) \, t^k\, .
$$
Thus we have
$$
b_k(s)= (-1)^{k+1} \frac{B_{k+1}(s)}{(k+1)!}\, ,
$$
and the related Bernoulli like numbers $b_k=b_k(0)$, and defining $b_{-1}=1$, thus
$$
\frac{1}{1-e^{-t}} =\sum_{k=-1}^{+\infty} b_k \, t^k\, .
$$

We can compute in two ways the expansion of $\displaystyle \frac{1-e^{-st}}{1-e^{-t}}$,
$$
\frac{1-e^{-st}}{1-e^{-t}} =\frac{1}{1-e^{-t}}-\frac{e^{-st}}{1-e^{-t}}=\sum_{k=0}^{+\infty} (b_k- b_k(s)) \, t^k
$$
and, for a positive integer $s$,
$$
\frac{1-e^{-st}}{1-e^{-t}} =\sum_{j=0}^{s-1}e^{-jt}=\sum_{k=0}^{+\infty} \frac{(-t)^k}{k!} 
\left (\sum_{j=0}^{s-1} j^k \right ).
$$
The sum above can be taken from $j=1$ to $s-1$, except in the case $k=0$ (in which case, we set $0^0=1$).

It follows the form of Faulhaber's formula for the sum of $k$-powers of the first $s-1$ integers.

\begin{proposition} We have
$$
\sum_{j=0}^{s-1} j^k = (-1)^k k! (b_k- b_k(s)).
$$
\end{proposition}

As it is well known we can compute the Bernoulli polynomials from the sequence of Bernoulli numbers:
$$
B_k(s)=\sum_{j=0}^k \binom{k}{j} \, B_{k-j} s^j\, .
$$
Hence we have
$$
b_k(s) =\sum_{j=0}^{k+1}  (-1)^j\, b_{k-j} \, \frac{s^j}{j!}
$$
and we conclude:

\begin{proposition} \label{prop:Faulhaber}
We have
$$
\sum_{j=0}^{s-1} j^k = (-1)^{k+1} k! \sum_{j=1}^{k+1} (-1)^j  b_{k-j} \, \frac{s^j}{j!}\, .
$$
\end{proposition}

\bigskip

\subsection*{Acknowledgements}
We thank the referee for his remarks and corrections that improved the article.

The first author has been partially supported by Comunidad de Madrid R+D Project PID/27-29 
and Ministerio de Ciencia e Innovaci\'on Project PID2020-118452GB-I00 (Spain). 

% \bigskip

\end{document}